\documentclass[a4paper,12pt]{article}
\usepackage[a4paper, margin=3.5cm, top=4cm]{geometry}
\usepackage{amssymb,amsmath,amsthm,thmtools,upref}
\usepackage[dvipsnames]{xcolor}
\usepackage[pdftex]{graphicx}
\usepackage{mathabx}
\usepackage{tikz}
\usepackage{wasysym}
\usepackage[all]{xy}
\usepackage{relsize}
\usepackage{pgf}

\usepackage{hyperref}

\usetikzlibrary{positioning}
\usetikzlibrary{cd}
\usetikzlibrary{intersections}

\newtheorem{theorem}{Theorem}[section]
\newtheorem{lemma}[theorem]{Lemma}
\newtheorem{proposition}[theorem]{Proposition}
\newtheorem{corollary}[theorem]{Corollary}

\theoremstyle{coloreddef}
\newtheorem{definition}[theorem]{Definition}
\newtheorem{proposition-definition}[theorem]{Proposition-Definition}

\theoremstyle{remark}
\newtheorem*{remark}{Remark}
\newtheorem*{example}{Example}
\newtheorem*{examples}{Examples}

\newtheorem*{conjecture}{Conjecture}

\theoremstyle{problem}

\newcommand{\oo}{\mathcal{O}}

\newcommand{\mc}{\mathcal}
\newcommand{\mbb}{\mathbb}

\DeclareMathOperator{\GL}{GL}

\DeclareMathOperator{\lv}{lv}

\definecolor{mypink1}{RGB}{231, 62, 156}
\definecolor{mypink2}{RGB}{219, 48, 122}
\definecolor{mygreen}{RGB}{107, 165, 51}
\definecolor{mygreen2}{RGB}{16,131,58}
\definecolor{teal}{RGB}{72, 146, 182}
\definecolor{lightpurple} {RGB}{172, 82, 204}
\definecolor{darkpurple}{RGB}{103,17,168}
\definecolor{myred1}{RGB}{128,0,32}
\definecolor{myblue1}{RGB}{24,189,225}
\definecolor{myblue2}{RGB}{16,25,144}

\begin{document}

\title{An approach to harmonic analysis on non-locally compact groups I: level structures over locally compact groups} \date{} \author{Raven Waller\thanks{This work was completed while the author was supported by an EPSRC Doctoral Training Grant at the University of Nottingham.}}

\maketitle

\begin{abstract}
We define a class of spaces on which one may generalise the notion of compactness following motivating examples from higher-dimensional number theory. We establish analogues of several well-known topological results (such as Tychonoff's Theorem) for such spaces. We also discuss several possible applications of this framework, including the theory of harmonic analysis on non-locally compact groups.
\end{abstract}

\section{Introduction} \label{sec:introduction}
The theory of harmonic analysis on locally compact groups is by now entirely classical. However, in higher dimensional number theory in particular, many objects arise which are no longer locally compact. For example, the field $\mbb{Q}_p (\!(t)\!)$ of formal Laurent series over $\mbb{Q}_p$ is not locally compact in any of the natural topologies which take into account both of its residue fields. The loss of local compactness for such fields is one of the most pervasive problems when one tries to study them.

In this paper we consider topological spaces whose topology may be reconstructed in a particular way from a locally compact group. This construction leads to a very natural generalisation of compactness, and in particular allows us to apply certain compactness arguments to groups related to higher dimensional local fields.

The main focus of this paper is thus the definition and properties of groups with a level structure. The main motivation for the definition comes from the study of two-dimensional local fields as follows. Consider once again $F=\mbb{Q}_p(\!(t)\!)$, and its rank two ring of integers $O_F = \mbb{Z}_p + t \mbb{Q}_p[\![ t]\!]$. One can define the level of a subset $S \subset F$ as the least integer $j$ such that $S$ contains a translate of a fractional ideal $p^it^j O_F$ for some $i\in \mbb{Z}$. The remarkable observation is then that, although $O_F$ is not compact, if one looks only at its open covers (in a particular topology) by sets of the same level (level $0$), all of them have finite subcovers. We thus have a weaker substitute which works ``on the level".

The definition of a level structure formalises this example, and gives a reasonable context in which to study such ``level compactness" properties. It turns out that many properties of compact sets are also shared by those which are only level-compact. For example, we obtain an analogue of Tychonoff's Theorem for products of compact spaces.

Since this text is dedicated to the development of a new, more general theory of compactness, there will be a substantial number of definitions given in quick succession. We try to give as much motivation as possible to show that each definition is important and, where possible, we demonstrate the kind of pathological cases that may arise if one doesn't take care to make the required assumptions.

The reader is thus asked to persevere with the abnormal Definition-to-Theorem ratio, if only because of the possible wide-reaching applications. Indeed, the notion of level-compactness is not at all limited to problems related to higher-dimensional number theory, and can almost certainly be studied in a variety of other contexts.

This paper is organised as follows. In Section \ref{sec:motivation} we recall the definition of a higher-dimensional local field, and explain the main motivating example for the constructions that follow. We then begin Section \ref{sec:levels} with the definition of a level structure on a group $G$. The remainder of the paper will be devoted to the study of such groups, and so it is paramount that the reader familiarises themselves with this definition as thoroughly as possible, keeping in mind the example of a higher dimensional local field from Section \ref{sec:motivation}.

We continue into \ref{sec:rigidity} with the notion of rigidity, and after several elementary results concerning levels we arrive as the definition of level-compactness. This is again a definition which the reader should take due time to become familiar with, as it is not only the main focus of the following sections, but possibly the most far-reaching idea in the entire text.

Following this, we work through several properties of level-compactness in \ref{sec:compactness}, including (though not limited to) many elementary topological properties which may be reformulated in this context. For example, one sees that sufficiently ``large" closed subsets of a level-compact set are level-compact, and that the product of level-compact spaces is level-compact. If nothing else, this section should allow the reader to become accustomed to working with all of the new definitions.

Finally, in Section \ref{sec:particulars} we discuss ways in which the theory developed in this text may be applied or further generalised. Indeed, although the author's main motivation for studying level structures comes from compactness problems in higher dimensional number theory, the notion of a level structure has many possible applications which are not at all related to compactness. Consider the following, for instance.

\begin{example}
Let $G= \mbb{Z}$ and let $X = \{ x \}$ be a one-element group. For $\gamma \in \mbb{Z}$, we define a collection of subsets of $G$ as follows. For $\gamma \ge 0$ we set $G_{ \{ x \}, \gamma} = \{ 0 \}$, and for $\gamma = -n$ with $n$ a positive integer we set $G_{ \{ x \}, \gamma} = \{ 0, 2, \dots, 2n-2 \}$. In the language of this paper, this defines a level structure for $G$ over $X$ of elevation $1$. This level structure defines a map $$- \lv : \{ \text{subsets of } G \} \rightarrow \mbb{Z},$$ which assigns to a subset $S$ of $G$ the length of the longest chain of an arithmetic progression of the form $\{ a, a+2, a+4, \dots, a+2r \}$ contained in $S$.
\end{example}

\begin{conjecture}
Let $k$ be an integer, and let $P_k = \{ p \text{ prime} : p \ge k \}$ be the set of all primes $\ge k$. Then $- \lv(P_k) = 2$ for all $k$.
\end{conjecture}

The above conjecture is a reformulation of the familiar Twin Prime Conjecture in the language of level structures. Note that we make no claims of being able to resolve this conjecture - it merely serves as an example that the framework of this paper is not restricted only to the confines of higher dimensional local fields.

\textbf{Acknowledgements.} I would like to thank Ivan Fesenko for his many comments on previous drafts of the current text, as well as various shorter works that were eventually incorporated here. I am also grateful to the many people with whom I have discussed this work during its various stages of completion - in particular Kyu-Hwan Lee, Sergey Oblezin, and Tom Oliver.

\section{Motivation from higher dimensional number theory} \label{sec:motivation}
We begin with a few motivating examples from higher dimensional number theory which illustrate the usefulness of the general constructions in this paper. First of all, recall the inductive definition of an $n$-dimensional local field.

\begin{definition}
An $n$-dimensional local field $F$ is defined inductively as follows. If $n=1$ then we take $F$ to be a local field (i.e. either complete discrete valuation field with finite residue field, or an archimedean field $F=\mbb{R}$ or $F=\mbb{C}$). For $n>1$ we then say that $F$ is an $n$-dimensional local field if it is a complete discrete valuation field whose residue field $\overline{F}$ is an $(n-1)$-dimensional local field. Finally, we define a $0$-dimensional local field to be a finite field.
\end{definition}

We will use the following indexing for residue fields. If $F$ is an $n$-dimensional local field, we write $F_{n-1}$ for the first residue field $\overline{F} = \oo_F / \mc{M}_F$, $F_{n-2}$ for the second residue field $\oo_{\overline{F}} / \mc{M}_{\overline{F}}$, and so on. With this convention, the field $F_i$ is an $i$-dimensional local field.

Recall that a system of local parameters for $F$ is a collection of elements $t_1, \dots, t_n \in O_F$ such that the residue of $t_i$ generates the maximal ideal of $\oo_{F_i}$.

\begin{definition}
Let $F$ be an $n$-dimensional local field. The rank $n$ ring of integers of $F$ is the subring $O_F$ of the ring of integers $\oo_F$ consisting of the elements $x \in \oo_F$ which remain integral under each of the residue maps $\oo_{F_i} \rightarrow F_{i-1}$ for $2 \le i \le n$.
\end{definition}

For $n>1$, an $n$-dimensional field $F$ can be endowed with a natural topology by lifting the topology of the $1$-dimensional residue field $F_1$ through the successive chain of residue fields. Under this topology (or with any of the other ``natural" topologies one may consider), $F$ is not locally compact, and so in particular there is no real-valued Haar measure on $F$. However, Fesenko noticed that by relaxing various conditions, it becomes possible to define a measure on such higher dimensional fields.

\begin{theorem} \label{fesenkomeasure}
Let $F$ be an $n$-dimensional local field with local parameters $t_1, \dots, t_n$. There exists a finitely additive, translation-invariant measure on the ring of subsets of $F$ generated by the sets $\alpha + t_1^{i_1} \cdots t_n^{i_n} O_F$ with $\alpha \in F$, $i_1, \dots, i_n \in \mbb{Z}$ which takes values in the field $\mbb{R} (\!( X_2)\!) \cdots (\!(X_n )\!)$.
\end{theorem}

\begin{proof}
See \cite{fesenko-aoas1}, \cite{fesenko-loop}.
\end{proof}

\begin{example}
Consider the two-dimensional field $F= \mbb{Q}_p(\!(t)\!)$. In this case we have local parameters $t_1 = p$ and $t_2=t$, and we have $O_F = \mbb{Z}_p + t \mbb{Q}_p [\![t]\!]$. The unique Fesenko measure $\mu$ on $F$ subject to the condition $\mu(O_F)=1$ satisfies $\mu(\alpha+ p^it^j O_F) = p^{-i} X^j$. This measure is countably additive except in very specific cases - see \cite{fesenko-aoas1} for details.
\end{example}

\begin{remark}
Let $F= \mbb{Q}_p(\!(t)\!)$ as in the above example, and let $\pi : \mbb{Q}_p [\![t]\!] \rightarrow \mbb{Q}_p$ be the residue map. If $U$ is a measurable subset of $\mbb{Q}_p$, the set $t^j \pi^{-1}(U)$ is measurable in $F$, and satisfies $\mu \left( t^j \pi^{-1}(U) \right) = X^j \mu_p (U)$, where $\mu_p$ denotes the Haar measure on $\mbb{Q}_p$ normalised so that $\mu(\mbb{Z}_p) = 1$.
\end{remark}

Since the existence of a Haar measure on a group $G$ is (roughly) equivalent to $G$ being locally compact, Theorem \ref{fesenkomeasure} thus suggests that a higher dimensional local field is ``not far" from being locally compact, in a sense which is to be made precise. This is further supported by the following.

\begin{proposition} \label{2DFcompactness}
Let $F$ be a two-dimensional nonarchimedean local field with parameters $t_1, t_2$. Then every covering of $t_1^{i} t_2^{j} O_F$ by sets of the form $\alpha + t_1^k t_2^j O_F$ admits a finite subcover.
\end{proposition}

\begin{proof}
Assume otherwise, i.e. there is such a cover $(V_m)_{m \in M}$ which admits no finite subcover. Since $t_1^{i} t_2^{j} O_F / t_1^{i_1+1} t_2^j O_F \simeq O_F / t_1 O_F$ is finite, there is $\theta_0 \in t_1^{i} t_2^{j} O_F$ with $\theta_0 + t_1^{i+1} t_2^{j} O_F$ not contained in a finite union of the $V_m$ (since otherwise we would have a finite subcover).

Similarly, there are $\theta_1, \dots, \theta_n \in t_1^{i} t_2^{j} O_F$ such that $\alpha_n + t_1^{i+n+1} t_2^j O_F= \theta_0 + \theta_1 t_1 + \dots + \theta_n t_1^n + t_1^{i+n+1} t_2^j O_F$ is not covered by a finite union of $V_m$. But since $F$ is complete, $\alpha = \lim_{n \rightarrow \infty} \alpha_n$ belongs to some $V_\ell$.

Now, $V_\ell = \beta + t_1^r t_2^j O_F$ for some $\beta \in F$, $r \in \mbb{Z}$. Furthermore, $\alpha \in V_\ell$ and $\alpha \in A_n = \alpha_n +  t_1^{i_1+n+1} t_2^{j} O_F$ for all $n \ge 0$. It is known from \cite{fesenko-aoas1} that such sets are closed under finite intersection, and so we have $V_\ell \cap A_n = V_\ell$ or $A_n$. If $V_\ell \cap A_n = A_n$ for every $n$, we have $\beta + t_1^r t_2^j O_F \subset \bigcap A_n = t_2^{j+1} \oo_F$, where $\oo_F$ is the rank one ring of integers of $F$, which is clearly impossible. We thus have $A_n \subset V_\ell$ for $n$ large enough. However, we have previously concluded that no $A_n$ can be covered by a finite number of the $V_m$, and so we have reached the desired contradiction.
\end{proof}

In other words, $F$ behaves like a locally compact space when we take covers which are of a ``similar size". We are thus prompted to look for a class of groups more general than those which are locally compact where one can perform harmonic analysis using Fesenko-type measures. This motivates the definition of a level structure.

\section{Level structures over locally compact groups} \label{sec:levels}
We now come to the most important definition of the entire text.

\begin{definition} \label{levelstructuredef}
Let $X$ be a locally compact topological group, and let $e \ge 0$ be an integer. A group $G$ is levelled over $X$ (with elevation $e$) if there is a collection $\mc{L}$ of subsets of $G$ satisfying the following conditions:

(1) Each element of $\mc{L}$ contains the identity element $e_G$ of $G$.

(2) $\mc{L}$ indexed by $\mc{U}(1) \times \mbb{Z}^e$, where $\mc{U}(1)$ is a basis of neighbourhoods of the identity in $X$ and $\mbb{Z}^e$ is lexicographically ordered from the right.

(3) For any $U, V \in \mc{U}(1)$ with $V \subset U$, if $G_{V,\gamma}, G_{U,\delta} \in \mc{L}$ with $\gamma \le \delta$ then $G_{V,\gamma} \cap G_{U,\delta} = G_{V,\delta}$.

(4) For any fixed $\gamma \in \mbb{Z}^e$, $G_{U,\gamma} \cup G_{V,\gamma} = G_{U \cup V, \gamma}$ and $G_{U,\gamma} \cap G_{V,\gamma} = G_{U \cap V, \gamma}$.

The collection $\mc{L}$ is called a level structure.
\end{definition}

\begin{remark}
If $U=V$, condition (3) simply says that $G_{U,\delta} \subset G_{U,\gamma}$ for $\gamma \le \delta$. For further discussion on the generalities of this definition, see Section \ref{sec:particulars}.
\end{remark}

Before we give examples, let us briefly discuss the importance of the conditions (1) to (4) in the definition above. The first two conditions mean that we are defining a local lifting of a basis of neighbourhoods at the identity of the base $X$. Moreover, by (2) this lifting is made up of a ``continuous" part (coming from $\mc{U}(1)$) and a ``discrete" part (coming from $\mbb{Z}^e$). The final two conditions then describe how the continuous and discrete parts of the structure should interact; (4) says that on any given discrete ``level" the local behaviour of $G$ should mimic the behaviour of $X$, while (3) says that the discrete component gives rise to a notion of relative size which respects the idea of ``size" encapsulated by the notion of subsets.

\begin{examples}
(1) Any locally compact group $G$ is levelled over itself with elevation $0$. In this case $\mc{L} = \mc{U}(1)$.

(2) An $n$-dimensional nonarchimedean local field $F$ is levelled over the one-element group with elevation $n$, where $\mc{L}$ consists of all principal fractional ideals of the rank-$n$ ring of integers $O_F$.

(3) Let $F$ be an $n$-dimensional local field (which may now be archimedean) and $F_1$ is its $(n-1)^{st}$ residue field. Then $F$ is levelled over $F_1$ with elevation $n-1$. If $F$ and $F_1$ have the same characteristic, so that $F$ is isomorphic to $F_1 (\!(t_2)\!) \dots (\!(t_n)\!)$, $\mc{L}$ consists of sets of the form $$t_2^{i_2} \dots t_n^{i_n}B(0,r) + \sum_{j=2}^n t_j^{i_j +1} t_{j+1}^{i_{j+1}} \dots t_n^{i_n} F_1 (\!(t_2)\!) \dots (\!(t_{j-1})\!)[\![ t_j]\!],$$ where $B(0,r)$ is the open ball of radius $r$ in $F_1$. In the mixed characteristic case, we associate to the pair $( t_1^{i_1} \oo_{F_1} , (i_2, \dots, i_n) )  \in \mc{U}(1) \times \mbb{Z}^{n-1}$ the set $t_1^{i_1} \dots t_n^{i_n} O_F \subset F$.

In the nonarchimedean case we may note that the neighbourhoods of the identity are themselves indexed by the totally ordered group $\mbb{Z}$; doing so recovers the previous example for such fields.

(4) Since the form taken by elements of $\mc{L}$ in the previous example may look quite complicated, we give a concrete example in dimension 4 to illustrate the general phenomenon. Let $F = \mbb{Q}_p (\!(t_2)\!) (\!(t_3)\!) (\!( t_4 )\!)$, so that $F_1 = \mbb{Q}_p$. The open balls in $F_1$ are then simply the fractional ideals $p^{i_1} \mbb{Z}_p$ for $i_1 \in \mbb{Z}$. Elements of $\mc{L}$ are thus of the form $$p^{i_1} t_2^{i_2}t_3^{i_3}t_4^{i_4} \mbb{Z}_p + t_2^{i_2 +1} t_3^{i_3}t_4^{i_4} \mbb{Q}_p [\![ t_2 ]\!] + t_3^{i_3+1} t_4^{i_4} \mbb{Q}_p (\!( t_2 )\!) [\![ t_3]\!] + t_4^{i_4+1} \mbb{Q}_p (\!( t_2 )\!) (\!( t_3 )\!) [\![ t_4]\!].$$ Note that this is exactly the set $p^{i_1}t_2^{i_2}t_3^{i_3}t_4^{i_4} O_F$.
\end{examples}

\begin{lemma} \label{interclosure}
Let $G$ be levelled over $X$ with elevation $e$. The set $\mc{L}$ is closed under finite intersection.
\end{lemma}

\begin{proof}
Let $G_{U, \gamma}, G_{V,\delta} \in \mc{L}$, and assume without loss of generality that $\gamma \le \delta$. In this case, $G_{V, \gamma} \cap G_{V, \delta} = G_{V, \delta}$ by condition (3), hence $G_{U, \gamma} \cap G_{V,\delta} = G_{U, \gamma} \cap G_{V, \gamma} \cap G_{V, \delta} = G_{U \cap V, \gamma} \cap G_{V, \delta}$ by condition (4). But $U \cap V \in \mc{U}(1)$ and is a subset of $V$, hence by (3) we have $G_{U \cap V, \gamma} \cap G_{V, \delta} = G_{U \cap V, \delta} \in \mc{L}$.
\end{proof}

By the above Lemma, the collection $\mc{L}$ is a basis of neighbourhoods of the identity for a topology on $G$.

\begin{definition}
We equip $G$ with the level topology as follows. We take $\mc{L}$ as a basis of neighbourhoods of the identity, and then extend to other points of $G$ by insisting that multiplication by any fixed element be continuous.
\end{definition}

If $G$, as in the first example, is a locally compact group viewed as being levelled over itself with $e=0$, this is just the original topology on $G$. On the other hand, if $G$ is (the additive group of) a two-dimensional local field as in the third example, the level topology on $G$ is \textit{not} the usual two-dimensional topology as defined (for example) in \cite{madunts-zhukov} - in this topology elements of $\mc{L}$ are closed but not open, for instance. 

\begin{definition}
An element of $G \mc{L} = \{ g H : g \in G, H \in \mc{L} \}$ is called a distinguished set. We also allow the empty set to be distinguished.
\end{definition}

\begin{example}
The distinguished subsets of a two-dimensional local field $F$ as defined by Fesenko in \cite{fesenko-aoas1} are exactly the distinguished sets of the elevation $1$ level structure of $F$ over its residue field, namely those of the form $\alpha + t_1^it_2^jO_F$.
\end{example}

The following result shows that the level topology is equivalent to the topology generated by the distinguished sets.

\begin{proposition} \label{fullintclosure}
$G \mc{L} \cup \{ \emptyset \}$ is closed under finite intersection.
\end{proposition}

\begin{proof}
We want to consider the intersection of $gG_{U,\gamma}$ and $h G_{V,\delta}$ with $g,h \in G$ and $G_{U,\gamma}, G_{V,\delta} \in \mc{L}$. If the intersection is empty we are done, so assume otherwise, so that there is an element $\alpha^{-1} \in G$ which is contained in the intersection. Translating by $\alpha$ then implies that $e_G$ is contained in both $\alpha g G_{U,\gamma}$ and $\alpha h G_{V,\delta}$. By continuity, both of these are basic open neighbourhoods of $e_G$ in the level topology, and hence by the above Lemma so is their intersection: $\alpha g G_{U,\gamma} \cap \alpha h G_{V,\delta} = G_{W,\beta}$. Translating back then gives $gG_{U,\gamma} \cap h G_{V,\delta} = \alpha^{-1} G_{W,\beta}$.
\end{proof}

The following definition won't be of immediate interest to us, but will be a useful tool to have, for example, if one wishes to construct an invariant measure on groups with level structure following \cite{fesenko-aoas1} and \cite{waller-GL2}.

\begin{definition}
The ring of ddd-sets of $G$ with respect to the level structure $\mc{L}$ is the minimal ring of sets containing $G \mc{L}$.
\end{definition}

We now arrive at the second most important definition of this text.

\begin{definition} \label{leveldef}
For $G_{U,\gamma} \in \mc{L}$, we define its level $$\lv(G_{U,\gamma}) = \max \{ \delta \in \mbb{Z}^e : G_{U,\gamma} \subset G_{U,\delta}  \}.$$ We then put $\lv(gG_{U,\gamma}) = \lv(G_{U,\gamma})$ for any $g \in G$.  For a general subset $S \subset G$, the level of $S$ (if it exists) is the minimal level of any subset of $S$ of the form $g G_{U,\gamma}$ for $g \in G$, $U \in \mc{U}(1)$, $\gamma \in \mbb{Z}^e$. We write $\lv(S)$ for the level of $S$.
\end{definition}

\begin{remark}
Note that we in fact have $$\lv(G_{U,\gamma}) = \max \{ \delta \in \mbb{Z}^e : G_{U,\gamma} \subset G_{U,\delta}  \} = \max \{ \delta \in \mbb{Z}^e : G_{U,\gamma} = G_{U,\delta}  \}.$$ One would like to simply define the level of $g G_{U,\gamma}$ to be $\gamma$, but since we have not ruled out the possibility that, say, $G_{U,\gamma} = G_{U,\delta}$ for $\gamma \neq \delta$, this would not be a consistent definition. The advantage of allowing such "degeneracy" is that one may define the induced level structure on a subgroup (which we will do in Section \ref{sec:induced}) with no additional difficulty.
\end{remark}

Clearly we have the equality $$\lv(S) = \min \{ \lv(S') : S' \subset S \text{ has a level} \}.$$ This leads to the following first observation concerning levels.

\begin{lemma} \label{levelsize}
If $\lv(A) < \lv(B)$ for two subsets $A$ and $B$ of $G$ then it cannot be the case that $A \subset B$. Moreover, there can be no $g \in G$ such that $A \subset gB$.
\end{lemma}

\begin{proof}
If $\lv(A) < \lv(B)$ then there is $\gamma \in \mbb{Z}^e$ such that $A$ contains a distinguished set $gG_{U,\gamma}$ but there is no $\delta \le \gamma$ with $h G_{V, \delta} \subset B$ for any choice of $h$ and $V$. In particular, there is at least one element of $gG_{U,\gamma}$ which is contained in $A$ but not in $B$. The second statement follows from the fact that level is invariant under the action of $G$ by translation.
\end{proof}

\begin{remark}
We may interpret level as being related to the size of a subset, with a higher level indicating a smaller size. Lemma \ref{levelsize} is in agreement with this interpretation. 
\end{remark}

\begin{example}
Let $F$ be a two-dimensional nonarchimedean local field with local parameters $t_1, t_2$, and let $\oo_F$ and $O_F$ be the rank-one and rank-two rings of integers of $F$. If $S \subset F$ is a finite set then $S$ does not have a level. For $i, j \in \mbb{Z}$ the set $t_1^i t_2^j O_F$ is a distinguished set of level $j$, while the set $t_2^j \oo_F$ is a non-distinguished set of level $j$.
\end{example}

One of the main purposes of this paper is to introduce the notion of ``level-compactness", for which the definition of level above will be very important. Using the intuitive interpretation of the level of a subset of $G$ as indicating its relative size, level-compactness will be equivalent to saying that any open cover by ``large enough" sets has a finite subcover (this will be made precise later, see Definition \ref{levelcompact}). However, our current definition of level is not quite strong enough for this, as the next example illustrates.

\begin{example}
Let $F$ be a two-dimensional nonarchimedean local field with $t_1$, $t_2$, $O_F$ as in the previous example. For all $\alpha, \beta \in F$ the set $A_{\alpha, \beta} = (\alpha + t_2 O_F) \cup (\beta + O_F)$ has level $0$. We have $O_F \subset \bigcup_\alpha A_{\alpha, \beta}$, where the union is over a complete (and infinite) set of representatives of $O_F / t_2 O_F$ in $O_F$. Taking any $\beta \notin O_F$, we obtain in this way an open cover of $O_F$ by sets of level $0$ with no finite subcover.
\end{example}

The problem in the above example is that, although the open cover we construct is essentially a cover of sets of level $1$ (since the $(\beta + O_F)$ components contribute nothing, being disjoint from $O_F$), the presence of this extra component formally lowers the level even though it does not contribute. This shows that in order to make reasonable progress towards any kind of ``compactness", we must consider only those open covers by sets which are ``uniformly large", which motivates the following definition.

\begin{definition} \label{uniformleveldef}
A subset $S \subset G$ has uniform level $\gamma$ if $\lv(S) = \gamma$ and for every point $s \in S$ there is a distinguished set $D_s$ of level $\gamma$ with $s \in D_s$ and $D_s \subset S$.
\end{definition}

It is immediate from the definition that any subset with uniform level is open. The sets $A_{\alpha, \beta} = (\alpha + t_2 O_F) \cup (\beta + O_F)$ in the previous example are not uniformly of level $0$, since $\alpha$ (for example) is not contained in any distinguished set of level $0$ lying inside $A_{\alpha, \beta}$. This additional condition is enough to eliminate such pathological examples; Proposition \ref{2DFcompactness} is simply the statement that $O_F$ is in fact compact with respect to open covers of uniform level $0$, formulated in more familiar terminology. We will see in the following section (Proposition \ref{Fcompactness}) the same result stated instead in the language of level structures.

\section{Rigidity and level-compactness} \label{sec:rigidity}
In the lead up to the third and final ``most important definition", we first examine a few more properties of the level of a subset. In particular, we would like this notion to be well behaved with respect to certain set theoretic operations. In order to achieve this, it is convenient to include the following rigidity assumption, which will also be very important in the next section when we look at properties of level-compactness.

\begin{definition}
A group $G$ levelled over $X$ with elevation $e$ is rigid if it satisfies the following condition: for any $\gamma \in \mbb{Z}^e$, if $G$ contains at least one subset of level $\gamma$ then $\lv(G_{U,\gamma}) = \gamma$ for all $U \in \mc{U}(1)$.
\end{definition}

\begin{remark}
Note that we always have $\lv(G_{U,\gamma}) \ge \gamma$ from the definition of level. The level of a subset was defined in a way as to allow for a consistent definition, for instance, in the case when there are no subsets having level $\le \delta$ for some $\delta$ by setting the level of all ``large enough" distinguished sets to be the maximum possible. In theory, this definition could also allow pathological cases where the level of a distinguished set no longer matches the intuitive notion of its ``size"; the notion of rigidity is intended to exclude these possible cases.
\end{remark}

We now begin our investigation of the interaction between levels and set operations with the following preliminary Lemma.

\begin{lemma} \label{smallcover}
Let $G$ be rigid, let $S \subset G$ with $\lv(S) = \gamma$, and for any $\delta \ge \gamma$ suppose that $G$ contains at least one subset of level $\delta$. Then for every $\delta \ge \gamma$ there is a subset $S' \subset S$ with $\lv(S') = \delta$. If, moreover, $S$ has uniform level $\gamma$, for each $s \in S$ there is a distinguished set $D_{s,\delta}$ of level $\delta$ with $s \in D_{s,\delta} \subset S$.
\end{lemma}

\begin{proof}
By rigidity and the definition of level, $S$ contains some distinguished set $gG_{U,\gamma}$ of level $\gamma$. By property (3) in the definition of level structure we have $gG_{U,\delta} \subset gG_{U,\gamma} \subset S$, and by rigidity we have $\lv(gG_{U,\delta}) = \delta$. This proves the first assertion.

Now suppose that $S$ has uniform level $\gamma$, and take any $s \in S$. Then there is some $g \in G$ such that $s \in gG_{U,\gamma} \subset S$. On the other hand, $s \in sG_{U,\delta}$, and from the proof of Proposition \ref{fullintclosure} we see that, if $\delta \ge \gamma$, $gG_{U,\gamma} \cap sG_{U,\delta}$ is a distinguished set $D_{s,\delta}$ which has level $\delta$, and $s \in D_{s,\delta} \subset S$ by construction.
\end{proof}

\begin{proposition} \label{intersectionlevel}
If $A$ and $B$ are both open subsets of $G$ which have a level and $A \cap B \neq \emptyset$ has a level, $\lv(A \cap B) \ge \max\{\lv(A), \lv(B)\}$. Furthermore, if $A$ and $B$ both have uniform level then so does $A \cap B$, and the inequality is in fact an equality.
\end{proposition}

\begin{proof}
If $A$ and $B$ are both distinguished sets then it follows from the proofs of Lemma \ref{interclosure} and Proposition \ref{fullintclosure} that $\lv(A \cap B) = \max\{\lv(A), \lv(B)\}$. In the general case, let $x \in A \cap B$. Let $D_A$ be a distinguished open neighbourhood of $x$ in $A$, and let $D_B$ be a distinguished open neighbourhood of $x$ in $B$. By Proposition \ref{fullintclosure} $D=D_A \cap D_B \subset A \cap B$ is a distinguished set, and by the first line it has level $\max\{\lv(D_A), \lv(D_B)\} \ge \max\{\lv(A), \lv(B)\}$.

It thus remains to show that $A \cap B$ cannot contain a distinguished set of any smaller level. Indeed, if it did contain such a set $D'$ of level less than $\max\{\lv(A), \lv(B)\}$, then both $A$ and $B$ would also contain $D'$, which is impossible since at least one of them has strictly larger level.

For the final assertion, note that if $A$ and $B$ have uniform level we may choose $D_A$ and $D_B$ to have levels $\lv(D_A) = \lv(A)$ and $\lv(D_B) = \lv(B)$, and then from what we have already proved it follows that $$\lv(D) = \max\{\lv(D_A), \lv(D_B)\} = \max\{\lv(A), \lv(B)\} \le \lv(A \cap B).$$ But since $D \subset A \cap B$ we have $\lv(D) \ge \lv(A \cap B)$ by definition, hence we have equality.
\end{proof}

\begin{remark}
Note that this Proposition \textit{not} prove that $\lv(A \cap B)$ exists in the non-uniform case, since $A \cap B$ may still contain distinguished sets of level $\gamma$ with $\max \{ \lv (A), \lv(B) \} < \gamma < \max \{ \lv(D_A), \lv(D_B) \}$. In the case that the elevation $e=1$, this is indeed enough to prove that the level exists, since there must be a minimal such $\gamma$, but for $e>1$ this is not necessarily the case, since bounded sequences in $\mbb{Z}^e$ do not necessarily have extrema.
\end{remark}

\begin{corollary}
Let $A$ and $B$ are subsets of $G$ which have a level such that $A \cap B$ has a level. If $\emph{int} A \cap \emph{int} B \neq \emptyset$ then $\lv(A \cap B) \ge \max\{\lv(A), \lv(B)\}$. (Here $\emph{int} S$ denotes the interior of a subset $S \subset G$.)
\end{corollary}

\begin{proof}
We can apply Proposition \ref{intersectionlevel}  to {int}$A$ and {int} $B$ to find a distinguished set $D$ inside $A \cap B$ of level $\lv(D) \ge \max \{ \lv( \text{int} A), \lv( \text{int} B ) \} \ge \max \{\lv(A), \lv(B) \}$. The same argument in the second paragraph of the proof of Proposition \ref{intersectionlevel} then shows that $A \cap B$ cannot contain a distinguished set whose level is lower than the latter.
\end{proof}

\begin{remark}
Unlike with intersections, the level of a subset is not at all well behaved under unions, and being of uniform level behaves even worse. (Indeed, the notion of uniform level was defined \textit{because} of the problems that unions may cause.) In certain specific cases it is possible to slightly control the behaviour of unions, but in general it is so wild that hardly anything may be said at all.
\end{remark}

We now come to the fundamental notion of level-compactness.

\begin{definition} \label{levelcompact}
Let $G$ be a group with level structure, and let $\gamma \in \mbb{Z}^e$. A subset $S \subset G$ is called $\gamma$-compact if every open cover (in the level topology) of $S$ by sets of uniform level $\gamma$ has a finite subcover. We will call $S$ level-compact if there is some $\gamma \in \mbb{Z}^e$ such that $S$ is $\gamma$-compact.
\end{definition}

\begin{remark}
As was hinted previously, it is important that each set in the cover has \textit{uniform} level $\gamma$. Note that although we refer to open covers in the definition (so that the reader may immediately see the connection with compactness), we may in fact omit the word "open" since we saw earlier that any set of uniform level is necessarily open.
\end{remark}

Possibly the most important example to keep in mind is the following, which is the main motivating example for the definition of level-compactness (and hence for this entire paper). Compare also with Proposition \ref{2DFcompactness}.

\begin{proposition} \label{Fcompactness}
Let $F$ be a $d$-dimensional nonarchimedean local field with parameters $t_1, \dots, t_d$. If $F$ is given the level structure of elevation $d-1$ over it's $1$-dimensional residue field then the subset $t_1^{i_1} \cdots t_d^{i_d} O_F$ is $\gamma$-compact with $\gamma = (i_2, \dots, i_d)$.
\end{proposition}

\begin{proof}
Assume otherwise, i.e. there is a $\gamma$-cover $(V_m)_{m \in M}$ which admits no finite subcover. Since $t_1^{i_1} \cdots t_d^{i_d} O_F / t_1^{i_1+1} \cdots t_d^{i_d} O_F \simeq O_F / t_1 O_F$ is finite, there is $\theta_0 \in t_1^{i_1} \cdots t_d^{i_d} O_F$ with $\theta_0 + t_1^{i_1+1} \cdots t_d^{i_d} O_F$ not contained in a finite union of the $V_m$ (since otherwise we would have a finite subcover).

Similarly, there are $\theta_1, \dots, \theta_n \in t_1^{i_1} \cdots t_d^{i_d} O_F$ such that $\alpha_n + t_1^{i_1+n+1} \cdots t_d^{i_d} O_F= \theta_0 + \theta_1 t_1 + \dots + \theta_n t_1^n + t_1^{i_1+n+1} \cdots t_d^{i_d} O_F$ is not covered by a finite union of $V_m$. But since $F$ is complete, $\alpha = \lim_{n \rightarrow \infty} \alpha_n$ belongs to some $V_\ell$.

Now, since $V_\ell$ has uniform level $\gamma$, there is a distinguished set $D$ with $\lv(D)=\gamma$ and $\alpha \in D \subset V_\ell$. On the other hand, we also have $\alpha \in A_n = \alpha_n +  t_1^{i_1+n+1} \cdots t_d^{i_d} O_F$ for $n \ge 0$. (Note that $A_n$ has uniform level $\gamma$.) By Lemma \ref{intpropertyforF}, the intersection of two distinguished sets is either empty or equal to one of them, and so we have $D \cap A_n = D$ or $A_n$. If $D \cap A_n = A_n$ for every $n$, we have $D \subset \bigcap A_n = t_2^{i_2+1} \dots t_d^{i_d} O^{(d-1)}_F$, where $O^{(d-1)}_F$ is the rank $(d-1)$ ring of integers of $F$. But then we have $\gamma = \lv(D) \ge \lv \left( \bigcap A_n \right) > \gamma$, which is a contradiction, hence we have $A_n \subset D \subset V_\ell$ for $n$ large enough. But we have previously concluded that no $A_n$ can be covered by a finite number of the $V_m$, and so we have reached the desired contradiction.
\end{proof}

\begin{remark}
It is essential that the elements of the cover all have the same level as $t_1^{i_1} \cdots t_d^{i_d} O_F$ in this Proposition. Indeed, $O_F = \bigcup \alpha + t_2O_F$ where $\alpha$ runs through an (infinite!) set of representatives for $O_F / t_2 O_F$, and since the union is disjoint there can be no finite subcover. It is equally important that the elements of the cover have uniform level, as we saw in an earlier example.
\end{remark}

\begin{remark}
The proof of Proposition \ref{Fcompactness} uses the fact that $F$ is complete in an essential way. As some of the consequences of completeness will be crucial later, it is worthwhile to ask if the completeness property (or perhaps a weaker alternative which still works for the above proof) can be restated purely in terms of the level structure.
\end{remark}

\begin{definition}
A group $G$ levelled over a locally compact group $X$ is called locally level-compact if for every $g \in G$ there is some $\gamma \in \mbb{Z}^e$ such that $g$ has a $\gamma$-compact neighbourhood.
\end{definition}

\begin{example}
An $n$-dimensional local field $F$ is locally level-compact over its local residue field.
\end{example}

\section{Properties of level-compactness} \label{sec:compactness}
In this section we will study various elementary topological properties of level-compactness. Since we may ask the question ``can we replace compactness by level-compactness" in almost every definition and theorem concerning compactness, we will of course not cover all possibilities here. Instead we focus as much as possible on results that are useful from the point of view of potential applications to areas of higher dimensional number theory and arithmetic geometry.

\subsection{Some elementary properties} \label{subsec:elementary}
\begin{proposition} \label{sublevel}
Let $G$ be levelled over $X$ with elevation $e$ with the level topology, suppose that $G$ is rigid, and suppose that $G$ contains at least one set of level $\gamma$. Then the following properties hold. \\
(1) If $S \subset G$ is $\gamma$-compact then $S$ is also $\delta$-compact for all $\delta \le \gamma$. \\
(2) If $S \subset G$ is $\gamma$-compact and $C$ is a closed subset of $S$ such that $S \backslash C$ is uniformly of level $\delta \le \gamma$ then $C$ is $\delta$-compact.
\end{proposition}

\begin{proof}
If $G$ has no subsets of level $\delta$ then the result trivially holds. Otherwise, let $S = \bigcup_m U_m$ be a uniform open $\delta$-cover of $S$. We may assume without loss of generality that each $U_m$ is a basic open set of level $\delta$. By rigidity, it follows from Lemma \ref{smallcover} that we may write $U_m = \bigcup_{\alpha \in U_m} V_\alpha$, where $V_\alpha$ is a distinguished set of level $\gamma$ containing $\alpha$. Then $S = \bigcup_{m,\alpha \in U_m} V_{\alpha}$ is a uniform open $\gamma$-cover of $S$, and by $\gamma$-compactness it has a finite subcover. It thus follows that for each $V_\alpha$ in this finite subcover we may take some $U_m$ containing it, and doing so gives a finite $\delta$-subcover of $S$.

The proof of (2) follows the same reasoning as the proof that closed subsets of compact spaces are compact. Indeed, first note that we know from (1) that $S$ is also $\delta$-compact. If we take any uniform open $\delta$-cover $C = \bigcup_m U_m$ of $C$, then $G = (G \backslash C) \cup \bigcup_m U_m$ is a uniform open $\delta$-cover of $S$, hence it has a finite subcover, and this gives us also a finite subcover of $C$.
\end{proof}

\begin{remark}
Property (2) may be thought of as saying that sufficiently small closed subsets of $\gamma$-compact sets are level-compact.
\end{remark}

It is important that $G$ contains a set of level $\gamma$ in the above Proposition. Indeed, if there are no sets of level $\gamma$ then every subset of $G$ is trivially $\gamma$-compact, in which case the result may not be true for some $\delta < \gamma$ where $\delta$-covers exist.

We can actually improve property (1) of Proposition \ref{sublevel} quite substantially. In order to do this, we first note that we may classify subsets that have no level into three distinct categories.

\begin{definition}
A subset $A \subset G$ is of type $S$ is there exists no distinguished subset $D$ with $D \subset A$.
\end{definition}

\begin{remark}
The subsets of type $S$ should be thought of as those which are ``too small" to have a level. In all examples we have given so far, finite sets have always been of type $S$.
\end{remark}

\begin{definition}
A subset $A \subset G$ is of type $L$ if for every $\gamma \in \mbb{Z}^e$ there is a distinguished set $D_\gamma$ of level $\delta \le \gamma$ with $D_\gamma \subset L$.
\end{definition}

\begin{remark}
Subsets of type $L$ are the opposite extreme to those of type $S$; they are the subsets which are ``too large" to have a level. If the map $G\mc{L} \rightarrow \mbb{Z}^e$ which sends every distinguished set to its level is surjective, the whole group $G$ is always of type $L$, although there may be more subsets of this type.
\end{remark}

\begin{definition}
Let $A$ be a subset of $G$ with no level. We say that $A$ is of type $E$ if it is not of type $S$ or of type $L$.
\end{definition}

\begin{remark}
The subsets of type $E$ are ``exceptional" subsets. If $G$ is of elevation $e \le 1$ and the level map is surjective then there are no subsets of type $E$. For $e=0$ this is trivial, and for $e=1$ it is an immediate consequence of the fact that any sequence in $\mbb{Z}$ which is bounded below has a minimum. When the level map is not surjective, many more sets of type $E$ may appear.
\end{remark}

We noted earlier that level-compactness should be thought of as being compactness with respect to open covers by ``sufficiently large" sets. The intuition from the above three definitions suggests that we should also attempt to allow sets of type $L$ in our covers. Of course, due to the same issues which appeared previously, we are guided towards the following subcollection of sets of type $L$.

\begin{definition}
Let $A \subset G$ be a subset of type $L$. For $\gamma \in \mbb{Z}^e$ we say that $A$ is $\gamma$-uniform if $A$ has an open covering $A = \bigcup U_i$ with each $U_i \subset A$ uniformly of level $\gamma$.
\end{definition}

The refinement of Proposition \ref{sublevel} is the following.

\begin{proposition} \label{biggersets}
Let $G$ be levelled over $X$ with elevation $e$, and suppose that $G$ is rigid. Let $A \subset G$ be $\gamma$-compact for some $\gamma \in \mbb{Z}^e$ such that $G$ contains a set of level $\gamma$, and let $A \subset \bigcup U_i$ be an open cover of $A$. Suppose that for each $i$ we have either (i) $U_i$ is uniformly of level $\gamma_i \le \gamma$, or (ii) $U_i$ is of type $L$ and is $\gamma$-uniform. Then $\bigcup U_i$ has a finite subcover.
\end{proposition}

\begin{proof}
If $U_i$ is of the form (i), we saw in the proof of Proposition \ref{sublevel} that we may cover $U_i$ by distinguished sets $\{ V_{\alpha}^{(i)} \}$ of uniform level $\gamma$. If $U_i$ is of the form (ii) then from the definition of $\gamma$-uniformity we also an open covering of $U_i$ by sets $\{ V_{\alpha}^{(i)} \}$ of uniform level $\gamma$. This gives us a uniform open $\gamma$-cover $A \subset \bigcup_{i,\alpha} V_{\alpha}^{(i)}$, which has a finite subcover by $\gamma$-compactness. For each $V_{\alpha}^{(i)}$ in this subcover, choosing one of the $U_i \supset V_{\alpha}^{(i)}$ gives the required subcover of $\bigcup U_i$.
\end{proof}

\begin{remark}
One may similarly define the notion of $\gamma$-uniformity for sets of type $E$ (note that for sets of type $S$ the condition can never be satisfied), and further refine Proposition \ref{biggersets} to include these sets as well. These exceptional sets of type $E$ seem quite mysterious, and it may be interesting to study their properties. There seems to be some link between the presence of exceptional sets and how badly $G$ can behave.
\end{remark}

We end this section with one final result on unions.

\begin{proposition} \label{compactunion}
Let $G$ be rigid. If $K_1 \subset G$ is $\gamma$-compact and $K_2 \subset G$ is $\delta$-compact with $\gamma \ge \lv(K_1)$, $\delta \ge \lv(K_2)$ then $K = K_1 \cup K_2$ is $\eta$-compact for some $\eta \ge \lv(K)$, if this level exists.
\end{proposition}

\begin{proof}
Let $\eta = \min \{ \gamma, \delta \}$. Then any uniform open $\eta$-cover of $K$ is an open cover of each of the $\eta$-compact sets $K_1$ and $K_2$, which both have finite subcovers. Taking the union of these subcovers then gives a finite subcover of $K$, hence $K$ is $\eta$-compact.

Since $K$ contains both $K_1$ and $K_2$, $$\lv(K) \le \min \{ \lv(K_1), \lv(K_2) \} \le \min \{ \gamma, \delta \} = \eta,$$ if the level of $K$ exists.
\end{proof}

\subsection{Product spaces} \label{subsec:products}
When speaking about compactness one also expects to consider products of spaces. If $G$ is levelled over $X$ and $H$ is levelled over $Y$ with the same elevation $e$, then $G \times H$ is naturally levelled over the product space $X \times Y$ with elevation $e$ via $(G \times H)_{(U \times V), \gamma} = G_{U,\gamma} \times H_{V,\gamma}$. The level topology on the product coincides with the product topology.

Since we will want to apply the earlier results of this section, the following easy Lemma is important.

\begin{lemma} \label{rigidproduct}
Let $G$ be levelled over $X$ and $H$ be levelled over $Y$, both with elevation $e$. If $G$ and $H$ are rigid then $G \times H$ is rigid over $X \times Y$.
\end{lemma}

\begin{proof}
We need to show that $\lv ( (G \times H)_{(U \times V), \gamma} ) = \gamma$. However, this follows from the definitions, since if $(G \times H)_{(U \times V), \gamma}$ is contained in some other $(G \times H)_{(U \times V), \delta}$ with $\delta > \gamma$ then (for example) $G_{U,\gamma} \subset G_{U,\delta}$, and so $\lv(G_{U,\gamma}) \ge \delta > \gamma$, which contradicts rigidity of $G$.
\end{proof}

\begin{remark}
Since the direct product of two spaces is commutative, the proof of Lemma \ref{rigidproduct} actually shows that $G \times H$ is rigid as long as at least one of $G$ and $H$ is rigid.
\end{remark}

\begin{proposition}
If $G$ is $\gamma_1$-compact and rigid over $X$, and if $H$ is $\gamma_2$-compact and rigid over $Y$, the product $G \times H$ is $\gamma$-compact over $X \times Y$, where $$\gamma = \min \{\gamma_1, \gamma_2 \}.$$
\end{proposition}

\begin{proof}
We know from Proposition \ref{sublevel} that $G$ and $H$ are both $\gamma$-compact. Thus if we take any uniform open $\gamma$-cover $G \times H = \bigcup_m (U_m \times V_m)$, we know that $G = \bigcup_m U_m$ has a finite subcover $G = \bigcup_{m \in M_1} U_m$ and $H=\bigcup_m V_m$ has a finite subcover $H= \bigcup_{m \in M_2} V_m$. It then follows that $\bigcup_{m \in M_1 \cup M_2} U_m \times V_m$ is a finite subcover of $G \times H$.
\end{proof}

\begin{remark}
This, along with the following results, is also true for the appropriate level-compact subsets of $G$ and $H$. However, for the sake of brevity we will formulate the statements only in terms of the full group $G$, and so on.
\end{remark}

The definitions also work for infinite products, but for the analogue of Tychonoff's theorem we will need to do a little more work.

\begin{lemma} \label{projcomp}
Let $G = \prod_{i \in I} G_i$ with each $G_i$ $\gamma$-compact over a space $X_i$. Then any open cover of $G$ by sets of the form $\pi_j^{-1}(U)$ with $U \subset G_j$ open of uniform level $\gamma$ has a finite subcover. (Here $\pi_j : G \rightarrow G_j$ is the projection map.)
\end{lemma}

\begin{proof}
Let $\mc{U}$ be such a cover, and let $\mc{U}_j$ be the collection of $U \subset G_j$ such that $\pi_j^{-1}(U) \in \mc{U}$. Suppose that there is no $j$ such that $\mc{U}_j$ covers $G_j$, so that for each $j$ we find $g_j \in G_j$ with $g_j$ not in the union of all elements of $\mc{U}_j$. But then $(g_i)_{i \in I}$ is not contained in any element of $\mc{U}$, which is not possible since this is a cover of $G$.

We can thus find some $j$ such that $\mc{U}_j$ is a cover of $G_j$, and by $\gamma$-compactness we can find a finite subcover $G_j \subset \bigcup_{k=1}^n U_k$, hence we have a finite subcover $G = \bigcup_{k=1}^n \pi_j^{-1} (U_k)$ of $\mc{U}$.
\end{proof}

Now we may prove a modification of the Alexander Subbase Theorem. For this, the set theory enthusiasts will note that we must assume the axiom of choice, since the proof requires the use of Zorn's Lemma.

\begin{lemma}
Let $G$ be levelled over $X$, and let $\mc{V}$ be a subbase for the level topology on $G$. If every collection of sets of uniform level $\gamma$ from $\mc{V}$ which covers $G$ has a finite subcover then $G$ is $\gamma$-compact.
\end{lemma}

\begin{proof}
Suppose that every such cover has a finite subcover, but $G$ is not $\gamma$-compact. Then the collection $\mc{F}$ of all open $\gamma$-covers of $G$ with no finite subcover is nonempty and is partially ordered by inclusion. Let $\{ E_\alpha \}$ be any totally ordered subset of $\mc{F}$, and put $E = \bigcup_\alpha E_\alpha$.

By definition, $E$ is a uniform open $\gamma$-cover of $G$. For any finite collection $U_1, \dots U_n$ of elements of the cover, we have $U_j \in E_{\alpha_j}$ for some $\alpha_j$, and since the $E_\alpha$ are totally ordered there is some $E_{\alpha_0}$ containing all of them. It follows that $E$ has no finite subcover, and so $E$ is an upper bound for $\{ E_\alpha \}$, thus by Zorn's Lemma there is a maximal element $\mc{M}$ of $\mc{F}$.

Let $S = \mc{V} \cap \mc{M}$, and suppose there is $g \in G$ that is not inside any element of $S$. Since $\mc{M}$ is a cover of $G$ there is some $U \in \mc{M}$ with $g \in U$, and since $\mc{V}$ is a subbase for the topology there are $V_1, \dots, V_n \in \mc{V}$ with $ g \in \bigcap_{i=1}^n V_i \subset U$. By assumption none of the $V_i$ are in $\mc{M}$, so by maximality $\mc{M} \cup \{ V_i \}$ contains a finite subcover $G=V_i \cup U_i$, where $U_i$ is a finite union of sets in $\mc{M}$. Then $U \cup \bigcup_{i=1}^n U_i \supset (\bigcap_{i=1}^n V_i) \cup (\bigcup_{i=1}^n U_i) \supset \bigcap_{i=1}^n (V_i \cup U_i) \supset G$. But this is a finite cover of $G$ by elements of $\mc{M}$, which contradicts $\mc{M}$ having no finite subcover.

It thus follows that $S$ is a cover of $G$, and since $S \subset \mc{V}$ it has a finite subcover by assumption. But $S$ is also contained in $\mc{M}$, and so it cannot have a finite subcover, hence the collection $\mc{F}$ must he empty, i.e. $G$ is $\gamma$-compact.
\end{proof}

This allows us to deduce the extension of Tychonoff's Theorem.

\begin{theorem} \label{tychonoff}
If $G = \prod_{i \in I} G_i$ with each $G_i$ $\gamma$-compact over $X_i$ then $G$ is $\gamma$-compact over $X = \prod_{i \in I} X_i$.
\end{theorem}

\begin{proof}
The collection $\{ \pi_j^{-1}(U_j) \}$ is a subbase for the product topology. By Lemma \ref{projcomp} any subcollection of this set which covers $G$ has a finite subcover, and then the Alexander Subbase Theorem shows that $G$ is $\gamma$-compact.
\end{proof}

\begin{remark}
If all of the $G_i$ are rigid, we don't need that all of the $G_i$ are $\gamma$-compact with the same $\gamma$. As long as there is a minimum $\gamma_{\min}$ across all $G_i$, this $\gamma_{\min}$ will be the (maximal) $\gamma$ that works for the product.
\end{remark}

For a collection of spaces $\{ A_i \}$ with subspaces $B_i \subset A_i$, recall that the restricted product $\prod ' A_i$ of the $A_i$ with respect to $B_i$ consists of sequences $(a_i)_i \in \prod A_i$ such that $a_i \in B_i$ for all but finitely many $i$. It follows from Theorem \ref{tychonoff} that the restricted product of a collection of rigid locally level-compact groups $\{ A_i \}$ with respect to a system of $\gamma$-compact subgroups $\{ B_i \}$ is again locally level-compact.

\begin{corollary}
Let $k$ be either a number field or the function field of a proper, smooth, connected curve over a finite field, and let $\mc{S}$ be an arithmetic surface which is a proper regular model of a smooth, projective, geometrically irreducible curve $X/k$. Then groups $\textbf{A}_{\mc{S}}$ of geometric adeles and $\mbb{A}_{\mc{S}_{'}}$ of analytic adeles, associated to $\mc{S}$ and a given set $\mc{S}_{'}$ of all fibres and finitely many horizontal curves on $\mc{S}$, are locally level-compact.
\end{corollary}

\begin{proof}
For a two-dimensional nonarchimedean local field, we saw in Proposition \ref{Fcompactness} that $t_1^i t_2^jO_F$ is $j$-compact, and since $t_2^j\oo_F$ is contained in all of the distinguished sets $t_1^i t_2^{j-1} \oo_F$ and disjoint from any of their $F$-translates it is trivially $(j-1)$-compact. For the archimedean components, the $j$- compactness of the subsets $\alpha + Ct^j + t^{j+1} F [\![ t ]\!]$ for $C \subset F$ compact follows immediately from the compactness of $C$, and so the archimedean fields $F (\!( t )\!)$ are also locally level compact. The adelic spaces are then restricted products of locally level-compact spaces with respect to level-compact subgroups (see \cite{fesenko-aoas2}). 
\end{proof}

\begin{remark}
The space $\textbf{A}_{\mc{S}}$ was first considered by Parshin, Beilinson and others, while $\mbb{A}_{\mc{S}_{'}}$ was first defined by Fesenko (see \cite{fesenko-aoas2} for the definitions, which are too lengthy to reproduce here). In particular when $\mc{S} = \mc{E}$ is the minimal regular model of an elliptic curve over a number field, the adelic spaces $\textbf{A}_{\mc{E}}$ and $\mbb{A}_{\mc{E}_{'}}$ are related to several important open problems in number theory and arithmetic geometry, including the Riemann Hypothesis and the Birch and Swinnerton-Dyer Conjecture. For details see \cite{fesenko-aoas2}, \cite{fesenko-aoas3}.
\end{remark}

\subsection{Discreteness} \label{subsec:discreteness}
We now consider connections with discreteness. It follows immediately from the definitions that a topological space which is both discrete and compact must be finite. Here we discuss what happens when we replace "compact" with "$\gamma$-compact". 

\begin{remark}
If every element of $\mc{L}$ is equal to $\{ e_G \}$, discreteness essentially forces the elevation $e$ to be equal to $0$, since all distinguished sets will have the same level. However, if there are non-singleton distinguished sets, this may not be the case, as the following example illustrates.
\end{remark}

\begin{example}
Let $X = \{ x \}$ be the single element group. We may endow $G = \mbb{Z}$ with an elevation $e=1$ level structure over $X$ via $G_{ \{ x \}, \gamma} = \{ 0 \}$ for $\gamma \ge 0$ and $G_{ \{ x \}, \gamma} = \{ 0, 1, \dots, n-1 \}$ for $\gamma = -n$ with $n$ a positive integer. The former sets all have level $-1$, while the latter have level $-n$. 

The level topology coincides with the discrete topology since $\{ 0 \} \in \mc{L}$, and in this case $- \lv(U)$ gives the length of the longest chain of consecutive integers in a subset $U$ of $G$. In this example, we easily see that $G$ has no subsets of type $S$, and that every finite set has a level. Since $G$ has elevation $e=1$, there are also no subsets of type $E$. We see that $G$ is of type $L$, but we may also construct infinitely many examples of proper subsets of type $L$. One such example is the set $$\{0, 2, 3, 5, 6, 7, 9, 10, 11, 12, 14, \dots \}$$ obtained by adding consecutive chains of longer lengths each time \textit{ad infinitum}.
\end{example}

\begin{remark}
Note that not all infinite subsets of $G$ in the above example are of type $L$: the set $2 \mbb{Z}$ of even integers, for example, has (uniform) level $-1$, while the set of all prime numbers has (non-uniform) level $-2$.
\end{remark}

\begin{remark}
Since, as we have noted now on several occasions, the level of a set is intuitively related to a notion of ``size", introducing various level structures on $\mbb{Z}$ may be of some interest in analytic number theory, as it gives an alternative way of defining the ``volume" of a set of primes. For example, in the introduction we stated the twin prime conjecture in terms of the level structure $G'_{ \{ x \}, \gamma} = 2 G_{ \{ x \}, \gamma}$, with $G_{ \{ x \}, \gamma}$ the level structure in the previous example.
\end{remark}

Suppose $G$ is levelled over $X$, and that $S \subset G$ is a $\gamma$-compact, discrete subset. By the results of the previous section, distinguished sets are the smallest open subsets of $G$, and so for every $s \in S$ there is a distinguished subset $D_s$ with $D_s \cap S = \{ s \}$.

Now, in general the distinguished sets $D_s$ may change wildly as $s$ varies, or simply all be of too high level for $\gamma$-compactness to come into play. However, if we have some control over these factors then we can indeed say something.

\begin{proposition}
Let $G$ be levelled over $X$ and let $S \subset G$ be discrete and $\gamma$-compact. If there exists a $\delta \le \gamma$ such that for every $s \in S$ there is a distinguished $D_s \subset G$ with $\lv(D_s) = \delta$ and $D_s \cap S = \{ s \}$ then $S$ is finite.
\end{proposition}

\begin{proof}
We have a uniform $\delta$ open cover $S \subset \bigcup_{s \in S} D_s$, and since $\delta \le \gamma$ $S$ is $\delta$-compact and so we have a finite subcover. However, since each $s \in S$ appears in exactly one element of the cover (namely $D_s$), this implies that $S$ must be finite.
\end{proof}

\begin{corollary}
If $S$ is discrete and $\gamma$-compact over $X$ for all $\gamma \in \mbb{Z}^e$, and furthermore $S$ can be completely disconnected by open subsets of some level $\delta$, then $S$ is finite.
\end{corollary}

\subsection{Compactness at all levels} \label{subsec:everywheregammacompact}
Clearly if $G$ is compact then it is $\gamma$-compact for all $\gamma$. However, it is not clear when (if at all) the converse holds. Obtaining results in this direction is almost equivalent to controlling the subsets of type $L$ and type $E$.

\begin{proposition}
Suppose $G$ is rigid and $\gamma$-compact for every $\gamma \in \mbb{Z}^e$, and that the following conditions are satisfied: \\
(i) $G$ contains no open subsets of type $E$; \\
(ii) For all $\gamma \in \mbb{Z}^e$ and for every open subset $A \subset G$ of type $L$, $A$ is $\gamma$-uniform.

Then any open cover $G=\bigcup U_i$ such that the set $\{ \lv(U_i) \}_i$ has an upper bound in $\mbb{Z}^e$ has a finite subcover. (Here we allow the possibility that $U_i$ has no level.)
\end{proposition}

\begin{proof}
Essentially all that we have to do is describe the possible structures of open subsets of $G$. If $U \subset G$ is open and does not have a level, it must be of type $L$ or of type $E$, since the only other possibility is that $U$ is of type $S$ and contains no nonempty open set (which is clearly false). Since by assumption no subset of type $E$ is open, the open subsets of $G$ either have a level or are of type $L$.

Now, let $G=\bigcup U_i$ be an open cover such that the levels of the $U_i$ are bounded above by $\gamma \in \mbb{Z}^e$. If $U_i$ is of type $L$, by (ii) we can cover $U_i$ by open sets $\{ V_{i,m} \}$ which are uniformly of level $\gamma$, and if $U_i$ is distinguished then we can do the same by Lemma \ref{smallcover}. This gives a uniform open $\gamma$-cover $G = \bigcup V_{i,m}$, from which we can take a finite subcover, and then take one $U_i$ containing each element of this subcover to obtain a finite subcover of $\{ U_i \}$.
\end{proof}

\subsection{The finite intersection property} \label{subsec:FIP}
An alternative characterisation of compactness can be formulated in terms of the finite intersection property. Recall that a family $\{ S_i \}$ of subsets of a topological space has the finite intersection property if any intersection of finitely many elements of the family is nonempty. It is well known that a topological space is compact if and only if every family of closed subsets satisfying the finite intersection property has nonempty intersection. Similarly, we have the following.

\begin{lemma}
$G$ is $\gamma$-compact if and only if any collection $\{E_i \}$ of closed sets with $G \backslash E_i$ uniformly of level $\gamma$ satisfying the finite intersection property has nonempty intersection.
\end{lemma}

\begin{proof}
First let $G$ be $\gamma$-compact, and suppose $\bigcap E_i = \emptyset$. Then $G = G \backslash \emptyset = G \backslash \bigcap_i E_i = \bigcup_i (G \backslash E_i)$ is a uniform open $\gamma$-cover of $G$, and so there is a finite subcover $G=\bigcup_{i=1}^n (G \backslash E_i) = G \backslash \bigcap_{i=1}^n E_i$. But the latter implies that $\bigcap_{i=1}^n E_i$ is empty, which is a contradiction.

Now suppose that any collection $\{E_i \}$ of closed sets with $G \backslash E_i$ uniformly of level $\gamma$ satisfying the finite intersection property has nonempty intersection. Let $G=\bigcup_i U_i$ be a uniform open $\gamma$-cover, and suppose there is no finite subcover. Then for any finite subcollection $U_1, \dots, U_n$ we have $\bigcap_{i=1}^n (G \backslash U_i) = G \backslash (\bigcup_{i=1}^n U_i) \neq \emptyset$, and so the collection $\{ G \backslash U_i \}$ satisfies the finite intersection property. We thus have $\emptyset \neq \bigcap_i (G \backslash U_i) = G \backslash (\bigcup_i U_i) = G \backslash G = \emptyset$, which is clearly impossible.
\end{proof}

\section{Induced level structures} \label{sec:induced}
Let $G$ be levelled over $X$, and let $H$ be a subgroup of $G$. One can define a level structure on $H$ over $X$ in a canonical way.

\begin{definition}
If $G$ has a level structure $\mc{L}$ over $X$, and $H$ is a subgroup of $G$, then the collection $\mc{L}_H = \{ H_{U,\gamma} = H \cap G_{U \gamma} : G_{U,\gamma} \in \mc{L} \}$ of subsets of $H$ is called the induced level structure.
\end{definition}

\begin{lemma}
The collection $\mc{L}_H$ is a level structure for $H$ over $X$.
\end{lemma}

\begin{proof}
Property (1) is satisfied since $H$ is a subgroup of $G$, and property (2) holds by construction. Properties (3) and (4) follow from the associativity and distributivity properties of union and intersection.
\end{proof}

It is clear from the definition that the level topology on $H$ given by $\mc{L}_H$ coincides with the topology induced on $H$ as a subspace of $G$.

Up until this point, it would have peen possible to restrict attention only to spaces where the level map $G \mc{L} \rightarrow \mbb{Z}^e$ is surjective. However, if we want induces level structures to make sense in general it is important that we do not make this assumption. For example, if $H$ is a subgroup of $G$ of level $\delta$, the induced level structure on $H$ contains no sets of level smaller than $\delta$.

\begin{proposition}
If $G$ is rigid and (locally) level-compact and $H$ is a closed subgroup of $G$ with the induced level structure such that $G \backslash H$ has a level in $G$ then $H$ is (locally) level-compact.
\end{proposition}

\begin{proof}
This follows immediately from Proposition \ref{sublevel}.
\end{proof}

\begin{corollary}
Any algebraic group $G$ over a higher dimensional local field $F$ which is closed in $\GL_n(F)$ such that $\GL_n(F) \backslash G$ has a level is locally level-compact.
\end{corollary}

In the case of algebraic subgroups, one may in fact consider an induced (partial) level structure over a more appropriate base than the original one.

\begin{proposition}
Suppose $G=\GL_m(F)$ for an $n$-dimensional nonarchimedean field $F$, with the partial level structure over $\GL_m(F_1)$ given by the distinguished subgroups $K_{\gamma_1, \dots, \gamma_n}$. Let $H$ be a subgroup of $G$ defined by finitely many polynomial equations $f_1, \dots, f_k \in O_F[\![X_1, \dots, X_{m^2}]\!]$, i.e. $$ H = \{ (g_{r,s}) \in G : f_i ( (g_{r,s} ) ) = 0, 1 \le i \le k  \}, $$ and let $\overline{H}$ be the subgroup of $\GL_m(F_1)$ defined by the reductions $$\bar f_1, \dots, \bar f_k \in F_1 [\![X_1, \dots, X_{m^2} ]\!].$$ If all of the polynomials $\bar f_i$ are separable (i.e. all the roots are simple), the association $(U \cap \overline{H} , \gamma) \mapsto H \cap G_{U,\gamma}$ defines a (partial) level structure for $H$ over $\overline{H}$ whose distinguished sets coincide with those of the induced level structure.
\end{proposition}

\begin{proof}
We must first check that the association is well-defined. In other words, if $\overline{H} \cap K_{\alpha} = \overline{H} \cap K_{\beta}$ for $\alpha, \beta \in \mbb{Z}$ we must make sure that $H \cap K_{\alpha, \gamma} = H \cap K_{\beta, \gamma}$ for all $\gamma \in \mbb{Z}^{n-1}$. Without loss of generality, we may assume that $\beta \le \alpha$, so that we have $H \cap K_{\alpha, \gamma} \subset H \cap K_{\beta, \gamma}$.

Let $g = (g_{r,s}) \in H \cap K_{\beta, 0}$. Then the reduction $\bar g = (\bar g_{r,s}) \in K_{\beta}$ is a root of all of the polynomials $\bar f_i$, and so we have $\bar g \in \overline{H} \cap K_{\beta} = \overline{H} \cap K_{\alpha}$. By Hensel's Lemma, there exists a unique $g' \in K_{\beta,0}$ with $\bar{g'} = \bar g$ and $f_i(g')=0$, and there exists a unique $g'' \in K_{\alpha,0}$ with $\bar{g'} = \bar g$ and $f_i(g'')= 0$. Since $f_i(g) = 0$, uniqueness of $g'$ forces $g=g'$. Furthermore, since $K_{\alpha, 0} \subset K_{\beta, 0}$, we have $g'' \in K_{\beta,0}$, and so uniqueness of $g'$ forces also $g'=g''$. We thus have $g = g'' \in H \cap K_{\alpha,0}$ and so we have the required equality for $\gamma =0$. The result for general $\gamma$ then follows from the fact that the map $I_m + t_1^{\gamma_1} \dots t_n^{\gamma_n} M \mapsto I_m + t_1^{\gamma_1} M$ is an isomorphism $K_{\gamma_1, \gamma} \rightarrow K_{\gamma_1, 0}$.

The fact that the distinguished sets coincide with those of the induced level structure is then immediate from the definition, and this implies that conditions (1) and (2) in the definition of a level structure hold. It remains to check (3) and (4).

Let $U,V \in \mc{U}(1)$ and $\gamma \le \delta \in \mbb{Z}^{n-1}$. We have $$G_{V \cap \overline{H}, \gamma} \cap G_{U \cap \overline{H}, \delta} = (H \cap G_{V, \gamma} ) \cap (H \cap G_{U, \delta}) = H \cap G_{V, \delta} = G_{V \cap \overline{H}, \delta},$$ and so (3) is satisfied.

Similarly, \begin{align*}
G_{U \cap \overline{H}, \gamma} \cup G_{V \cap \overline{H}, \gamma} &= (H \cap G_{U, \gamma}) \cup (H \cap G_{V, \gamma}) \\
&= H \cap (G_{U, \gamma} \cup G_{V, \gamma}) \\
&= H \cap G_{U \cup V, \gamma} = G_{(U \cup V) \cap \overline{H}, \gamma},\end{align*} and \begin{align*}
G_{U \cap \overline{H}, \gamma} \cap G_{V \cap \overline{H}, \gamma} &= (H \cap G_{U, \gamma}) \cap (H \cap G_{V, \gamma}) \\
&= H \cap (G_{U, \gamma} \cap G_{V, \gamma}) \\
&= H \cap G_{U \cap V, \gamma} \\
&= G_{(U \cap V) \cap \overline{H}, \gamma}, \end{align*} and so (4) is satisfied.
\end{proof}

\begin{remark}
It seems to be a very common theme concerning groups with level structure that, while the topological and analytic properties of $G$ should be essentially self-contained (since the information is bound to the level structure, which is in principle just a collection of subsets of $G$), by choosing the correct base $X$ for the level structure we may see various analogies between $G$ and $X$ which makes certain properties of $G$ appear more clearly.
\end{remark}

One may, for example, utilise the induced level structures to study the above remark. If we consider $G$ with the same level structure over two different bases $X$ and $X'$, we may take the product to obtain a level structure for $G \times G$ over $ X \times X'$. The induced level structure on the diagonal image of $G$ inside $G \times G$ will then be related to both $X$ and $X'$.

To end this section, we return briefly to the problem of degeneracy for induced level structures. Consider the following example.

\begin{example}
Consider $G= \mbb{Q}_p(\!(t)\!)$ with its usual level structure of elevation $e=1$ over $\mbb{Q}_p$. We consider the induced level structure on $H=\mbb{Q}_p$. For $G_{i, j} =  p^i t^j \mbb{Z}_p + t^{j+1} \mbb{Q}_p [\![ t]\!] \in \mc{L}$ we have $H_{i,j} = 0$ for $j > 0$, $H_{i,0} = p^i \mbb{Z}_p$, and $H_{i,j} = \mbb{Q}_p$ for $j<0$. One sees at once that the induced level topology coincides with the $p$-adic topology, and that the level structure has essentially fallen away completely. It this example it is indeed true that $j$-compactness for all $j \in \mbb{Z}$ implies compactness, since one only needs to check the $j=0$ case.
\end{example}

In the above example we see that (in a certain sense) the elevation may drop when inducing a level structure on a subgroup. To make this more precise, we first consider how to inflate the elevation.

\begin{proposition} \label{inflate}
Let $G$ be levelled over $X$ with elevation $e$, and let $\mc{L}$ be the level structure. Then for any $w \in \mbb{Z}$ the collection $\mc{L}'$ of subsets of $G$ indexed by $\mc{U}(1) \times \mbb{Z}^{e+1} \simeq \mc{U}(1) \times \mbb{Z}^e \times \mbb{Z}$ given by $G_{U,\gamma,j} = G$ for $j<w$, $G_{U,\gamma,w} = G_{U,\gamma}$, and $G_{U,\gamma,j} = \{ e_G \}$ for $j>w$ defines a level structure for $G$ over $X$ of elevation $e+1$.
\end{proposition}

\begin{proof}
It follows immediately from the definitions and the fact that we order lexicographically from the right.
\end{proof}

By iterating the above procedure we may inflate the elevation from $e$ to $e+n$ for any positive integer $n$. While inflation may not produce anything particularly interesting, it does allow us to be precise about saying that the level structure on $\mbb{Q}_p$ induced from $\mbb{Q}_p(\!(t)\!)$ is essentially of elevation $0$. It also allows us to take products of groups with different elevations, since we may inflate everywhere to the maximal elevation among the components.

\section{A few particulars and generalities} \label{sec:particulars}
We end with a few ideas that aren't necessarily important for the overall theme of the text, but may nonetheless be interesting to think about.

First, the reader will surely have noted that, we do not use anywhere the fact that the base space $X$ of a group with level structure is locally compact. One may thus relax the requirements on the base $X$ to the following much less stringent conditions: $X$ must be a pointed set, and the set must come equipped with a collection $\mc{U}(1)$ of subsets which contain the distinguished point of $X$ and satisfies the defining condition to be a ``basis of neighbourhoods" at this point. (The set $X$ itself need not be a topological space, it just needs to look like one around the distinguished point.)

One possibly important example is to take $X$ to be a group levelled over a second space $X'$. 

\begin{lemma} \label{levelstack}
Let $X$ have level structure $\mc{L}_X$ over $X'$ with elevation $e_1$, and let $G$ be levelled over $X$ with elevation $e_2$ and level structure $\mc{L}_G$ indexed by $\mc{L}_X \times \mbb{Z}^{e_2}$. Then $G$ is levelled over $X'$ with elevation $e_1 + e_2$ and level structure $\mc{L}_G$.
\end{lemma}

\begin{proof}
To an open neighbourhood $U$ of $1$ in $X'$ and to $(\gamma, \delta) \in \mbb{Z}^{e_1+e_2}$ with $\gamma \in \mbb{Z}^{e_1}$ and $\delta \in \mbb{Z}^{e_2}$ we associate the distinguished set $G_{U, (\gamma, \delta)} =G_{X_{U, \gamma}, \delta}$. The fact that this defines a level structure immediately follows from the fact that $\mc{L}_X$ and $\mc{L}_G$ are both level structures.
\end{proof}

\begin{remark}
In an upcoming paper, we will construct an invariant measure on a certain class of groups with level structure, and in this case we do use the fact that $X$ is locally compact. This is one of the main reasons that the requirement was written into the original definition.
\end{remark}

While considering generalities regarding the base space, it would also be interesting to investigate also the special case where the elevation $e=1$. In this case, the indexing group is simply $\mbb{Z}$ (or, if we only take a partial level structure, a subset of this). In particular, we have the following elementary Lemma which may have some interesting consequences.

\begin{lemma} \label{Zmin}
Any sequence in $\mbb{Z}$ which is bounded below has a minimum element.
\end{lemma}

We have already noted this previously in passing, when we remarked that no subsets of type $E$ can exist in the case $e=1$ due to this property when the level map is surjective. Clearly the Lemma is not true for higher powers of $\mbb{Z}$: the sequence $(-n, 0)_{n \ge 0}$ is bounded below by $(0,-1)$ since we order from the right, but it is strictly decreasing. This may have important consequences regarding the difference between arithmetic geometry in dimension $2$ ($e=1$ case) and in dimensions $\ge 3$ ($e \ge 2$ case).

We should also consider further how much information about $G$ we can extract from the base. We noted earlier that, while various properties are intrinsic to $G$, a convenient choice of base can make these properties easier to study.

One possibility in this direction (for which the author thanks Tom Oliver for the suggestion!) is the following. Consider the two-dimensional local field $F = \mbb{Q}_p (\!(t)\!)$. We obtain the natural level structure of $F$ over the residue field $\overline{F} = \mbb{Q}_p$ via the association $(p^i \mbb{Z}_p, j) \mapsto p^i t^j O_F$. Alternatively, we obtain the same level structure on $F$ over the field $\mbb{F}_p (\!(t)\!)$ via the association $(t^j \mbb{F}_p [\![ t]\!] , i) \mapsto p^i t^j O_F$. Depending on how much information may be ``shared" between a group and its base, it may thus be possible to obtain new analogies between the two different local fields $\mbb{Q}_p$ and $\mbb{F}_p (\!(t)\!)$. Alternatively, by going up one more level to a three-dimensional field, it may be interesting to investigate possible analogies between two-dimensional fields that arise in a similar way.

The information contained within the base also has some ramifications for the potential study of harmonic analysis. For example, for $F$ an $n$-dimensional local field, it is possible to construct a finitely additive, translation-invariant measure on $F$ (see \cite{fesenko-aoas1}, \cite{fesenko-loop}, \cite{morrow-2dint}) and on $GL_2(F)$ (see \cite{morrow-gln}, \cite{waller-GL2}). In each case, it was noted that the measure was (in an appropriate sense) a lift of the measure on the residue field. 

The initial expectation was that this should always be the case. More precisely, one can conjecture that for a rigid, locally level-compact $G$ levelled over a locally compact group $X$ we should have such a measure $\mu_G$ on $G$ which satisfies (for example) $\mu_G ( G_{U,\gamma} ) = \mu_X(U) \textbf{Y}^\gamma$, where $\mu_X$ is a Haar measure on $X$.

Unfortunately, we know that in general this cannot be the case. Indeed, we saw previously that for an arithmetic surface $\mc{S}$ both the geometric and analytic adelic spaces $\textbf{A}_{\mc{S}}$ and $\mbb{A}_{\mc{S}_{'}}$ may be given a level structure of elevation $1$ over a locally compact group. However, Fesenko showed in \cite{fesenko-aoas2} that there can be no measure on $\textbf{A}_{\mc{S}}$ which appropriately lifts the one-dimensional adelic measure.

On the other hand, the space $\mbb{A}_{\mc{S}_{'}}$ was constructed exactly so the analytic adelic measure \textit{does} lift the one-dimensional measure. It would thus be interesting to consider possible conditions on a general group $G$ levelled over a locally compact $X$ which allow such lifting properties. One possible possible starting point is to examine the difference in the level-compactness properties of the two different adelic spaces above, which will have some relation to the fact that the rank one ring of integers $\oo_F$ of a two-dimensional field is level-compact but not $\lv(\oo_F)$-compact.

Finally, we note that the current text was written mainly in order to formalise certain behaviours of higher dimensional local fields. To this end, our ``justification" for many of the definitions and additional assumptions here is that they hold for such fields and groups closely related to them. In particular, we have not been particularly diligent in trying to find examples of spaces which do not satisfy the properties we require.

It would thus be interesting to see whether some of our definitions may in fact be turned into theorems, and if this is not the case one should see some interesting counterexamples. For example, in almost all of the current paper we assume that our groups are rigid in order to manipulate levels in the correct way, but it may be possible that our version of rigidity can be deduced from a weaker assumption.

In any case, the author would be interested to see constructions of interesting examples of spaces with level structure which are not related to higher dimensional local fields.


\begin{thebibliography}{MZ95}

\bibitem[F03]{fesenko-aoas1}
I.~Fesenko.
\newblock Analysis on arithmetic schemes. {I}.
\newblock {\em Docum. Math. Extra Kato Volume}, pages 261--284, 2003.

\bibitem[F05]{fesenko-loop}
I.~Fesenko.
\newblock Measure, integration and elements of harmonic analysis on generalized
  loop spaces.
\newblock {\em Proc. of St. Petersburg Math. Soc.}, 12:179--199, 2005.

\bibitem[F10]{fesenko-aoas2}
I.~Fesenko.
\newblock Analysis on arithmetic schemes. {II}.
\newblock {\em J. K-theory}, pages 437--557, 2010.

\bibitem[F17]{fesenko-aoas3}
Ivan Fesenko.
\newblock Analysis on arithmetic schemes. {III}.
\newblock 2017.

\bibitem[M08]{morrow-gln}
Matthew Morrow.
\newblock Integration on product spaces and {$GL_n$} of a valuation field over
  a local field.
\newblock {\em Communications in Number Theory and Physics}, vol. 2(no.
  3):563--592, 2008.

\bibitem[M10]{morrow-2dint}
Matthew Morrow.
\newblock Integration on valuation fields over local fields.
\newblock {\em Tokyo J. Math}, 33(1):235--281, 2010.

\bibitem[MZ95]{madunts-zhukov}
A.~I. Madunts and I.~B. Zhukov.
\newblock Multidimensional complete fields: topology and other basic
  constructions.
\newblock {\em Amer. Math. Soc. Transl. (Ser. 2)}, (166):1--34, 1995.

\bibitem[vU]{wester-fourier}
Wester van Urk.
\newblock Fourier analysis on higher local fields.
\newblock To appear.

\bibitem[W18]{waller-GL2}
Raven Waller.
\newblock Measure and integration on {$GL_2$} over a two-dimensional local
  field.
\newblock {\em To appear}, 2018.

\end{thebibliography}
\end{document}